\documentclass[12pt,a4paper]{amsart}
\usepackage{graphics}
\usepackage{epsfig}
\usepackage{graphicx}
\theoremstyle{plain}
\usepackage{amssymb}
\advance\hoffset-20mm \advance\textwidth40mm

\newtheorem{theorem}{Theorem}
\newtheorem{lemma}{Lemma}
\newtheorem*{theo*}{Theorem}
\newtheorem{proposition}{Proposition}
\newtheorem{corollary}{Corollary}
\theoremstyle{definition}

\newtheorem*{definition*}{Definition}
\newtheorem{example}{Example}
\newtheorem{remark}{Remark}

\def\ad{\mathop{\rm ad}}

\def\Der{{\rm Der}}
\def\Dim{{\rm dim}}
\begin{document}
\sloppy
\title[On   nilpotent and solvable Lie algebras of  derivations ]
{On   nilpotent and solvable Lie algebras \\ of   derivations }
\author
{ Ie. O. Makedonskyi and  A.P.  Petravchuk}
\address{Ievgen  Makedonskyi:
Department of Algebra and Mathematical Logic, Faculty of Mechanics and Mathematics, Kyiv
Taras Shevchenko University, 64, Volodymyrska street, 01033  Kyiv, Ukraine}
\email{makedonskyi@univ.kiev.ua , makedonskyi.e@gmail.com}
\address{Anatoliy P. Petravchuk:
Department of Algebra and Mathematical Logic, Faculty of Mechanics and Mathematics, Kyiv
Taras Shevchenko University, 64, Volodymyrska street, 01033  Kyiv, Ukraine}
\email{aptr@univ.kiev.ua , apetrav@gmail.com}
\date{\today}

%\thanks{
% The second author was partially supported by the DFFD,
%grant F28.1/026}
%
\begin{abstract}
Let $\mathbb K$ be a field and $A$ be a commutative associative $\mathbb K$-algebra which is an integral domain. The Lie algebra $\Der A$ of all $\mathbb K$-derivations of $A$ is an $A$-module in a natural way and if $R$ is the quotient field of $A$ then $R\Der A$ is a vector space over $R.$  It is proved that if  $L$ is a nilpotent  subalgebra of $R\Der A$ of  rank $k$ over $R$ (i.e. such that $\dim _{R}RL=k $), then  the  derived length of $L$ is at most $k$ and $L$ is finite dimensional over its field of constants. In case of  solvable  Lie algebras over a field of characteristic zero their  derived length does not  exceed $2k.$ Nilpotent and solvable Lie algebras of rank $1$ and $2$ (over $R$) from the Lie algebra $R\Der A$ are characterized. As a consequence we obtain  the same estimations for nilpotent and solvable Lie algebras of vector fields with polynomial, rational, or formal coefficients.
\end{abstract}
\maketitle

%%%%%%%%%%%%%%%%%%%%%%%%%%%%%%%%%%%%%%%%%%%%%%%%%%%%%%%%%%%

\section*{Introduction}
Let $\mathbb{K}$ be a  field and $A$ be an associative commutative $\mathbb K$-algebra with identity, without zero  divisors, i.e. an integral domain. The set  $\Der A$  of all $\mathbb K$-derivations of $A,$  i.e. $\mathbb K$-linear operators $D$ on $A$ satisfying the Leibniz's rule: $D(ab)=D(a)b+aD(b)$ for all $a, b\in A$ is a Lie algebra over $\mathbb K$  and an $A$-module in a natural way:  given $a\in A, D\in \Der A,$   the derivation $aD$ sends any element $x\in A$ to $a\cdot D(x).$  The structure of the Lie algebra $\Der A$ is of great interest because,  in geometric terms,   derivations can  be considered as  vector fields on geometric objects. For example, in case $\mathbb K=\mathbb C$ and $A=\mathbb C[x_{1}, \ldots , x_{n}],$ the polynomial ring, any $D\in \Der A$ is of the form $$D=f_{1}\frac{\partial}{\partial x_{1}}+\cdots +f_{n}\frac{\partial}{\partial x_{n}}, \ f_{i}\in \mathbb C[x_{1}, \ldots , x_{n}],$$
i.e. $D$ is a vector field on $\mathbb C^{n}$ with polynomial coefficients.
Lie algebras  of vector fields with polynomial, formal power series, or analytical coefficients were  studied intensively
by many authors (see, for example, \cite{Lie}, \cite{Bavula}, \cite{Draisma}, \cite{Olver},   \cite{Olver1},  \cite{PBNL}, \cite{Post}).

In general case, when $A$ is an integral domain,  subalgebras $L$ of $\Der A$ such that $L$ are submodules of the $A$-module $\Der A$  were studied in \cite{J1}, and  sufficient conditions were given for $L$ to be simple. In this paper, we study subalgebras of the Lie algebra $\Der A$ at the other extreme: nilpotent and solvable, under the condition that they are of finite rank over $A.$
Recall that if $R$ is the quotient field  of $A,$  then the rank ${\rm rk}_{R}L$ is defined as ${\rm rk}_{R}L=\dim _{R}RL.$
Any subalgebra $L$ of the Lie algebra $\Der A$ determines uniquely the field $F=F(L)$ of constants consisting of all $r\in R$ such that $D(r)=0$ for all $D\in L.$ The vector space  $FL$ over the field $F$ is actually  a Lie algebra over $F$ (note that $RL$ being a Lie algebra over $\mathbb K$ is not in general a Lie algebra over $R$).  The main results of the paper:
if $L$ is a nilpotent subalgebra of the Lie algebra $R\Der A$ with ${\rm rk}_{R}L=k,$ then the derived length of $L$ is at most $k$ and the Lie algebra $FL$ is finite dimensional over $F$ (Theorem \ref{nilpot}).  In case when $L$ is solvable of rank $k$ and ${\rm char}  \  \mathbb K=0,$ the derived length does not exceed $2k$ (Theorem \ref{derlensol}). If $\dim _{\mathbb K}L<\infty ,$ then the last estimation can be improved to $k+1.$
If we apply these  results to the important case $\mathbb K=\mathbb C$  and $A=\mathbb C[[x_{1}, \ldots , x_{n}]],$  the ring of formal power series, we
obtain that nilpotent subalgebras of the Lie algebra $\Der A$ of rank $k$ over $R$ have derived length $\leq k$ and solvable subalgebras have derived length $\leq 2k.$  Note that in this particular case it was proved in \cite{Mart1}  that all nilpotent subalgebras have derived length at most $n$ and solvable at most $2n$ (see Corollary \ref{Martelo}).

We also give a rough characterization   of nilpotent and solvable subalgebras of rank $1$ and $2$ over $R$ from the Lie algebra $R\Der A$ (over their  fields of constants).
Such a characterization   can be applied to study finite dimensional Lie algebras of smooth vector fields in three variables (the case of one and two variables was studied in \cite{Lie}, \cite{Olver}, and \cite{Olver1}). Using the same approach we gave in \cite{MP12} a description of finite dimensional subalgebras of $W(A)$ in case $A=\mathbb K(x, y),$ the field of rational functions.

We use standard  notations, the ground field $\mathbb{K}$ is  arbitrary unless  otherwise  stated. The quotient field  of the integral domain $A$ under consideration will be denoted by $R.$ Any derivation $D$ of $A$ can be uniquely extended  to a derivation of $R$ by the rule: $D(a/b)=(D(a)b-aD(b))/b^{2}.$  It is obvious that $R\Der A$ is a subalgebra of the Lie algebra $\Der R$ and $\Der A$ is embedded in a natural way into $R\Der A.$  We will  denote $R\Der A$ by $W(A),$ it is a vector space over $R$ of dimension ${\rm rk}_{R}\Der A,$  and  a Lie algebra over $\mathbb K$ but not in general case over $R.$
All subspaces and subalgebras of $W(A)$ will be  considered over the field $\mathbb K$ unless otherwise stated.
If $L$ is a subalgebra of the Lie algebra $W(A),$ then the field
$F=\{r\in R \ | \ D(r)=0 \ \mbox{for all} \ D\in L\}$  will be called the {\it field of constants} of  $L$. We denote by $s(L)$ the derived length of a (solvable) Lie algebra $L.$  If a Lie algebra $L$ contains an ideal $N$ and a subalgebra $B$ such that $L=N+B, N\cap B=0,$  then we write $L=B\rightthreetimes N$ for the semidirect sum of $B$ and $N.$

\section{Nilpotent  subalgebras of finite rank of the Lie algebra $W(A)$}
We will use the next statement which can be immediately checked.
\begin{lemma}\label{properties}
Let $D_{1}, D_{2}\in W(A)$  and $a, b\in R.$ Then it holds:

{\rm 1.}  $[aD_{1}, bD_{2}]=ab[ D_1, D_2]+aD_1(b)D_2-bD_2(a)D_{1}.$

{\rm 2.} If $a, b\in \ker D_{1}\cap \ker D_{2},$  then  $[aD_{1}, bD_{2}]=ab[ D_1, D_2].$
\end{lemma}

Let $L$ be a nonzero  subalgebra of rank $k$ over $R$ of the Lie algebra $W(A)$
and let $\lbrace D_1, \dots , D_k \rbrace$ be a basis of $L$ over $R$. Recall that the set
$RL$  of $W(A)$  consists of all linear combinations
of elements $aD$, where $a \in R$, $D \in L;$ analogously one can define the set $FL$ (F=F(L) is the field of constants of $L$).

\begin{lemma}
Let $L$ be a nonzero subalgebra of  $W(A)$ and let $FL$, $RL$ be  $\mathbb{K}$-spaces  defined as above.
Then:

{\rm 1.}  $FL$ and $RL$ are $\mathbb{K}$-subalgebras of the Lie algebra $W(A)$.
Moreover,  $FL$ is a Lie algebra over the field $F.$

{\rm 2.} If the algebra $L$ is  abelian,  nilpotent, or  solvable  then the Lie algebra $FL$ has the same property, respectively.
\end{lemma}

\begin{proof}
Immediate check.
\end{proof}

\begin{lemma}
Let $L$ be a subalgebra of finite rank over $R$  of the Lie algebra $W(A)$, $Z=Z(L)$ be the center of $L$
and $F$ be the field of constants of $L$. Then ${\rm rk}_{R} Z = \dim _{F}FZ$
and $FZ$ is a subalgebra of the center $Z(FL)$. In particular, if
$L$ is abelian, then $FL$ is an abelian subalgebra  of  $W(A)$ and ${\rm rk}_{R}L=\dim _{F}FL.$

\end{lemma}

\begin{proof}
Let $\lbrace D_1, \dots , D_k \rbrace$ be  a basis of $Z$ over $R$.
Take any element $D \in Z$ and write $D=a_1D_1+ \dots +a_kD_k,$ where $a_{i}\in R.$.
Then  for any element $S\in L$ we have :
\[0=[S,D]=[S,a_1D_1+ \dots +a_kD_k]=S(a_1)D_1+ \dots +S(a_k)D_k.\]
Since the elements $D_1, \dots , D_k$ are linearly independent over $R$ it follows
from the last relation that $S(a_i)=0, i=1,  \dots , k$. Hence $a_i \in F, i=1,  \dots , k$
and $\lbrace D_1, \dots , D_k \rbrace$ is a basis of $FL$ over $F$.
The latter means that ${\rm rk}_{R}Z= \dim_{F}FZ$.
\end{proof}

\begin{lemma} \label{intersectionideal}
Let $L$ be a  subalgebra of  the Lie algebra $W(A)$ and  $I$ be an ideal of $L$. Then
the vector  space $RI\cap L$ (over $\mathbb K$)  is also  an ideal  of $L.$
\end{lemma}

\begin{proof}
Take  any element $\sum_{k=1}^m r_k i_k\in RI\cap L$ with  $r_k \in R$, $i_{k}\in I$, $k=1, \dots , m$.
Then for an arbitrary element $D \in L$  we obtain:
\[[D,\sum_{k=1}^m r_k i_k]=\sum_{k=1}^m (D(r_{k})i_{k}+r_{k}[D, i_{k}])\in RI\cap L.\]
This completes the proof of Lemma.
\end{proof}

\begin{lemma} \label{nilbasis}
Let $L$ be a nilpotent   subalgebra of rank $k>0$ over $R$  of  the Lie algebra $W(A).$
Then:

{\rm 1.}  $L$ contains a series of ideals
\begin{equation}\label{chain}
0=I_0 \subset I_1 \subset I_2 \subset \dots \subset I_k=L
\end{equation}
such that ${\rm rk}_{R} I_s=s$, $s=0,  \dots , k.$

{\rm 2.}  $L$ possesses an $R$-basis $\{ D_1,  \dots , D_k\}$ such that
$I_s=L \cap (RD_1 +\cdots +RD_s)$, $s=1, \dots ,  k$ and $[L, D_s]\subset I_{s-1}.$

{\rm 3. } $\dim _{F} FL/FI_{k-1}=1.$

{\rm 4. } $[I_j, I_j]\subset I_{j-1}$, $j=1, \ldots , k.$

\end{lemma}

\begin{proof}
1-2. Take a nonzero element $D_1 \in Z(L)$ and  put $I_1 = RD_1 \cap L$. Then $I_1$ is an ideal of $L$
by Lemma \ref{intersectionideal}. Assume that we have built the set of
elements $D_1, \dots , D_j$ such that the $\mathbb K$-spaces $I_s=L \cap (RD_1 +\dots + RD_s)$, $s=1, \dots , j$
are ideals of the Lie algebra $L$ and  $[L, D_s]\subset I_{s-1}$ for  $s=1, \dots , j$ with  ${\rm rk}_{R}{I_s}=s$.
Take a  one-dimensional ideal $\langle D_{j+1}\rangle +I_{j}$ of the (nilpotent) quotient algebra $L/I_{j}.$
Then  $[L, D_{j+1}] \subset I_j$ and the elements $D_{1}, \ldots , D_{j+1}$ are linearly independent  over $R.$
Put $I_{j+1}=L \cap (RD_1 +\dots +RD_{j+1})$. Then $I_{j+1}$ is an ideal of $L$ by Lemma \ref{intersectionideal} and
 ${\rm rk}_{R} I_{j+1}=j+1.$ Therefore we obtain  by induction  the chain (\ref{chain}) of ideals  and
  a basis $\{ D_{1}, \ldots , D_{k}\}$ of $L.$ This basis satisfies obviously the condition 2 of Lemma.

3. Take  an arbitrary element $D=a_1D_1+ \dots +a_kD_k \in L$ and any element $D_{i}$ from the basis $\{ D_{1}, \ldots , D_{k}\}.$  Then using Lemma \ref{properties} we get:
\[[D_i, \sum_{j=1}^{k}a_{j}D_{j}]=\sum_{j=1}^{k}D_{i}(a_{j})D_{j}+\sum_{j=1}^{k}a_{j}[D_{i}, D_{j}]\]
Since $[D_{i}, I_{s}]\subseteq I_{s-1}$ we see from the last relation that
 $D_i(a_k)=0, i=1, \ldots ,k$. The latter means  $a_k \in F$ and therefore $\dim _{F}FL/FI_{k-1}=1.$
 The part 3 of Lemma  is proved.

4. Let $S=\sum _{i=1}^{j}a_{i}D_{i}$ and $T=\sum _{i=1}^{j}b_{i}D_{i}$ be any elements of $I_{j}.$
Then:
 \begin{equation}\label{Dj1}
 [S, T]= [\sum _{i=1}^{j}a_{i}D_{i}, \sum _{i=1}^{j}b_{i}D_{i}]=\end{equation}
  \[  =[\sum _{i=1}^{j-1}a_{i}D_{i}, \sum _{i=1}^{j-1}b_{i}D_{i}]+[\sum _{i=1}^{j-1}a_{i}D_{i}, b_{j}D_{j}]+
+ [a_{j}D_{j}, \sum _{i=1}^{j-1}b_{i}D_{i}]+[a_{j}D_{j}, b_{j}D_{j}]. \]
The first three summands in the right side of the last equality lie  in $RD_{1}+\ldots +RD_{j-1}$ by the part 2 of this Lemma.
Let us show that $[a_{j}D_{j}, b_{j}D_{j}]=0.$ Really, applying the part 3 of this Lemma to the Lie algebra $I_{j}$ (instead of $L$) we see that $\dim _{F_{j}}F_{j}I_{j}/F_{j}I_{j-1}=1,$ where $F_{j}$ is the field of constants of $I_{j}.$
 Therefore the elements  $a_{j}$ and $b_{j}$  are linearly dependent over $F_{j},$  say, $\alpha a_{j}+\beta b_{j}=0$  for some $\alpha , \beta \in F_{j}.$ If $\alpha \not= 0,$ then $a_{j}=-\alpha ^{-1}\beta b_{j}$ and  $[a_{j}D_{j}, b_{j}D_{j}]=[-\alpha ^{-1}\beta b_{j}D_{j}, b_{j}D_{j}]=0$ because $\alpha ^{-1}\beta \in F_{j}\subseteq ker D_{j}.$  Analogously
$[a_{j}D_{j}, b_{j}D_{j}]=0$ provided that $\beta \not= 0.$ Thus, the right side of the equality (\ref{Dj1}) lies in $I_{j-1}.$ The proof of part 4  is completed.
\end{proof}

\begin{corollary}\label{derlennil}
Let $L$ be a  nilpotent subalgebra of rank $k$ over $R$  of  the Lie algebra $W(A).$
Then the derived length of  $L$ is at most $k.$
\end{corollary}

\begin{proof}
See part 4 of  Lemma \ref{nilbasis}.
\end{proof}
\begin{remark}\label{subspaces}
We will  use the next almost obvious statement:  If $V$ is a vector space over the field $\mathbb K$ and $U, W$ are subspaces of $V$ of finite codimension,  then the subspace $U\cap W$ is also of finite codimension in $V.$
\end{remark}

\begin{lemma} \label{nillinalg}
Let $V$ be a vector space over the  field $\mathbb{K}$ and $L$ be a finite dimensional
$\mathbb{K}$-subspace of ${\rm End}(V).$  Suppose that $L$ acts nilpotently on $V$ (i. e. $L^n(V)=0$ for some $n \geq 1$).
If the vector space  $V_{0}=\{ v\in V \ |  \ Lv=0\}$ is finite dimensional over $\mathbb{K},$ then $\dim V<\infty$.
\end{lemma}

\begin{proof}
Induction on the smallest number $n$ such that $L^n(V)=0$. If $n=1$ then $LV=0$,
$V=V_{0}$, hence $\Dim(V)<\infty$ by the conditions of Lemma. Consider the
$\mathbb{K}$-subspace $U=L(V)$ of $V$. The vector space  $U_{0}=\{ u\in U \ |  \ Lu=0\}$
has obviously finite dimension over $\mathbb{K}$ and $L^{n-1}(U)=0$.
By the inductive assumption $\dim U < \infty .$ Choose a basis $\{g_1, g_2, \dots , g_k \}$ of $L$ over $\mathbb{K}$.
It follows from the above proven that $\dim V/\ker g_{i}<\infty$ because the linear operator $g_{i}$ maps $V$ into $U$ and $\dim U<\infty .$  But then $\dim V/ \cap _{i=1}^{k}\ker g_{i}<\infty$ by Remark \ref{subspaces} and therefore $V_{0}$ is of finite codimension in $V.$ Since $\dim V_{0}<\infty$ by the conditions of Lemma, we obtain $\dim V<\infty.$
\end{proof}

\begin{theorem}\label{nilpot}
Let $L$ be a nilpotent subalgebra of finite rank over $R$ from the Lie algebra $W(A)$   and $F=F(L)$ be the field of constants of $L.$
Then the Lie algebra $FL$ is finite dimensional over $F$.
\end{theorem}

\begin{proof}
Let $k={\rm rk}_R L$ and  $0=I_{0}\subseteq I_{1}\subseteq \ldots \subseteq I_{k}=L$ be the series of ideals of $L,$ constructed in Lemma  \ref{nilbasis}. Take a basis $\{D_1,  \dots ,D_k\}$ of $L$ over $R$ obtained in  such a way as in Lemma \ref{nilbasis}.  We  prove by induction on $i$ that $\dim _{F}FL/FI_{k-i} < \infty$.
It is true for $i=1$ by Lemma \ref{nilbasis}, part (3). Assume that $\dim _{F}FL/FI_j < \infty$ for $j=k-i$
and consider the natural action (by multiplication) of $FL$ on the $F$-space $V=FI_j/FI_{j-1}$.
It holds $[FI_j,FI_j] \subset FI_{j-1}$ by Lemma \ref{nilbasis} and  therefore
$FI_jV=0$. Hence $V$ is a module over the  finite dimensional (over $F$) Lie
algebra $FL/FI_j$. The Lie algebra  $FL/FI_j$ acts nilpotently on $V$ because the algebra $FL$ is  nilpotent.
Let $V_{0}=\{ v\in V \ | \ (FL/FI_{j})v=0\}$
and $D=a_1D_1 + \dots+ a_jD_j$ a representative of  an arbitrary element from $V_{0}\subseteq V=FI_{j}/FI_{j-1}.$
Then for any $i=1, \dots , k$ we have  $$[D_i, D]=[D_i, a_1D_1 + \dots a_jD_j]=$$
$$=[D_i, a_1D_1 + \dots a_{j-1}D_{j-1}]+a_j[D_i, D_j]+D_i(a_j)D_j\in I_{j-1}.$$
 The first and second summands in the right side of the last equality  lie in $I_{j-1},$
 so $D_i(a_j)D_j\in I_{j-1}.$ The latter means  that  $D_i(a_j)=0, i=1, \ldots , k$ and therefore $a_j \in F$ by definition of the  field $F.$
Thus $\dim_{F}V_{0} =1$ and  Lemma \ref{nillinalg} yields $\dim _{F}FI_{j}/FI_{j-1}<\infty .$ But then $\dim _{F}FL/FI_{j-1}=\dim _{F}FL/FI_{k-(i+1)}<\infty .$  When $i=k$ we obtain the inequality $\dim _{F}FL<\infty .$
\end{proof}

\begin{proposition} \label{nilfewdim}
Let $L$ be a nilpotent subalgebra of $W(A)$ and $F=F(L)$ be its field of
constants. Then:

{\rm 1.}  If ${\rm rk}_{R}L=1,$ then $L$ is abelian and $\dim_{F}FL=1.$

{\rm 2.}  If ${\rm rk}_{R}L=2,$ then there exist elements $D_1, D_2 \in FL$ and $a \in R$ such that
\[ FL=F\langle D_1, aD_1, \ldots , \frac{a^k}{k!} D_{1}, D_{2}\rangle , \  k\geq 0  \ (\mbox{if} \ k=0, \ \mbox{then put} \ FL=F\langle D_{1}, D_{2}\rangle ), \]
where $[D_1, D_2]=0$, $D_1(a)=0$, $D_2(a)=1.$
\end{proposition}

\begin{proof}
 1. It follows from Lemma \ref{nilbasis}, part 3.

   2. Let ${\rm rk} _{R}L=2.$  Suppose that  ${\rm rk} _{R}Z(L)=2 $ and let $\{ D_{1}, D_{2}\}$ be a basis of $Z(L)$ over $R.$ Put $I_{k}=RD_{k}\cap L, \ k=1, 2.$ Since $I_{1}\cap I_{2}=0$ and $\dim _{F}FL/FI_{k}=1, k=1, 2$ by Lemma \ref{nilbasis} we see that $\dim _{F}FL=2$ and $FL=F\langle D_{1}, D_{2}\ | \ [D_{1}, D_{2}]=0\rangle$  is of type 2 of Lemma. Let now ${\rm rk} _{R}Z(L)=1,$  $D_{1}\in Z(L)$ be a nonzero element and $I_{1}=RD_{1}\cap L.$ Then $I_{1}$ is an ideal of $L$ and $\dim _{F}FL/FI_{1}=1$   by Lemma \ref{nilbasis}. Choose any nonzero element $D_{2}\in L\setminus I_{1}.$ The elements $D_{1} , D_{2}$ form a basis of $L$ over $R$ and $[D_{1}, D_{2}]=0.$ Since the Lie algebra $L$ is nilpotent the operator  $\ad D_{2}$ acts nilpotently on the abelian ideal $FI_{1}$ of the algebra $FL$ over the field $F.$

Let us show that $\ad D_{2}$ has in some basis of $FI_{1}$ (over $F$) the matrix which is a single Jordan block. Really, any Jordan chain  for $\ad D_{2}$ on $FI_{1}$ contains an element of the form $aD_{1}$ such that $[D_{2}, aD_{1}]=0.$ But then $D_{2}(a)=0$ and taking into account the equality $D_{1}(a)=0$ we get $a\in F.$  The latter means that $\ad D_{2}$ has the only Jordan chain $\{D_{1}, a_{1}D_{1}, \ldots , a_{k}D_{1}\}$ on $I_{1}$ with $a_{i}\in R$  and its matrix in this  basis  is a single Jordan block. Since  $[D_{2}, D_{1}]=0,  [D_{2}, a_{1}D_{1}]=D_{1}, \ldots , [D_{2}, a_{k}D_{1}]=a_{k-1}D_{1}$  we have
$$D_{2}(a_{1})=1, D_{2}(a_{2})=a_{1}, \ldots , D_{2}(a_{k})=a_{k-1}.$$
Denoting $a=a_{1}$ we obtain $D_{2}(a_{2}-a^{2}/2!)=0.$ Since also $D_{1}(a_{2}-a^{2}/2!)=0$ we get  $a_{2}-a^{2}/2!\in F.$  But then without loss of generality we can take $a_{2}=a^{2}/2!.$  Repeating these considerations we obtain a basis $\{ D_{1}, aD_{1}, \ldots  , (a^{n}/n!)D_{1}\}$ of the ideal $FI_{1}.$
\end{proof}

\begin{remark}
Let $A=\mathbb K[x_{1}, \ldots , x_{n}]$ be the polynomial algebra in $n$ variables  over $\mathbb K.$ Then $\Der A={W}_n(\mathbb{K})$ is the Lie algebra of all vector fields on $\mathbb K^{n}$ with polynomial  coefficients.  Take the elements $D_{1}=\frac{\partial}{\partial x_{1}},  D_{2}=\frac{\partial}{\partial x_{2}}$  from the Lie algebra ${W}_n(\mathbb{K})$ and put $a=x_{2}\in A.$  It is obvious that  the Lie algebra
$$ L_{n}=F\langle \frac{\partial}{\partial x_{1}}, x_{2}\frac{\partial}{\partial x_{1}}, \ldots , (x_{2}^{n}/n!)\frac{\partial}{\partial x_{1}}, \frac{\partial}{\partial x_{2}}\rangle ,$$  where $F=\mathbb K(x_{3}, \ldots , x_{n})$   is the field of constants for $L,$ is nilpotent of nilpotency class $n-2.$
If we consider the union $L=\cup_{i=1}^{\infty}L_{i}$ of the ascending chain of Lie algebras $L_{1}\subset L_{2}\subset L_{3}\subset \ldots \subset L_{n}\subset \ldots ,$ then the algebra $L$ is solvable of derived length $2$ and infinite dimensional over $F.$
\end{remark}

\section{Solvable subalgebras of $W(A)$}

\begin{lemma}\label{solfewdim}
Let $L$ be a solvable subalgebra of rank $1$ over $R$ from the Lie algebra $W(A)$
and $F=F(L)$ be its  field of constants. Then:

{\rm 1.}  If $L$ is abelian,  then $FL$ is one dimensional over $F.$

{\rm  2.}   If $L$ is nonabelian, then $\dim _{F}FL=2.$ In particular, $s(L)=2.$
\end{lemma}

\begin{proof}
1.  Let $L$ be abelian. Since $L$ is nilpotent of rank $1$ over $R,$ it follows from Proposition \ref{nilfewdim} that $\dim _{F}FL=1.$

2. Suppose that $L$ is nonabelian and take a maximal (by inclusion) abelian ideal $I\subset L$ and a nonzero element $D_{1}\in I.$ Then $FI=FD_{1}$ is of dimension $1$ over $F$ by the above proven. Choose any two elements $b_{1}D_{1}, b_{2}D_{1}\in L\setminus I$  (recall that all elements of $L$ are of the form $aD_{1}$ for some $a\in R$).
Since $I$ is a maximal abelian ideal we have $C_{L}(I)=I$ and therefore $[D_{1}, b_{i}D_{1}]=D_{1}(b_{i})D_{1}\not= 0, i=1,2.$ Denoting  $D_{1}(b_{i})=a_{i}, i=1, 2$ we obtain from the last relations that $a_{1}, a_{2}$  are nonzero elements of the field $F=\ker D_{1}.$ But then $D_{1}(a_{1}^{-1}b_{1}-a_{2}^{-1}b_{2})=0$ and therefore $a_{1}^{-1}b_{1}-a_{2}^{-1}b_{2}\in F.$ The latter means that the elements $b_{1}D_{1}, b_{2}D_{1}$ are linearly dependent over $F$ and  $FL$ is nonabelian of dimension $2$ over $F.$
\end{proof}
\begin{remark}
How to construct solvable subalgebras of rank $1$ from $W(A)$? The answer is as following: to build an abelian Lie algebra one should take any $\mathbb K$-subspace $V$ from the subfield $\ker D_{1}$ and set $L=VD_{1}.$ Then $L$ is abelian and every abelian Lie algebra of rank $1$ over $R$ can be obtained in such a way.  To construct a nonabelian Lie algebra one should take a derivation $D_{1}$ possessing  an element $b\in R$ such that $D_{1}(b)=1.$
Then $L$ is a subalgebra of the Lie algebra $(\ker D_{1}) D_{1}+(b\ker D_{1})D_{1}.$  The latter  Lie algebra is isomorphic to the general affine Lie algebra $ga_{1}(\ker D_{1}).$
%Note that the problem of classifying of subalgebras from  $ga_{1}(\ker D_{1})$ is wild as a problem from Linear %Algebra.
\end{remark}

\begin{lemma}\label{quotient}
Let $L$ be a solvable subalgebra of rank $k$ over $R$ from the Lie algebra $W(A),$  $F=F(L)$  be its field of constants  and $I$  an ideal
of $L$ such that $I=RI \cap L.$ If  ${\rm rk}_{R}I=k-1,$  then $\dim _{F}FL/FI\leq 2,$
  in particular, $s(L/I)\leq 2.$  Besides, if $\dim _{F}FL/FI=2,$ then $s(L/I)=2.$
\end{lemma}
\begin{proof}
Take an $R$-basis $\{ D_1, \dots , D_k\}$ of $L$ such that the elements $ D_1, \dots , D_{k-1}$ form  an $R$-basis of $I$.
Consider the following $\mathbb K$-subspace $M \subset RL$:
$$M=\lbrace a_kD_k  \ | \ \exists a_1, \ldots , a_{k-1} \ \mbox{with} \  a_1D_1 + \cdots +a_{k-1}D_{k-1}+a_kD_k \in L\rbrace .$$
It is easy to see that $M$ is a subalgebra of rank $1$ over $R$ from  the Lie algebra $RL.$ Since the subalgebra $M $  has derived length $\leq 2$ by Lemma \ref{solfewdim} and $L/I\simeq M$  we get that $s(L/I)\leq 2.$

Take any nonzero abelian ideal $J/I$ of $L/I.$  Any element $D\in J\setminus I$ can be written in the form  $D=a_{1}D_{1}+\ldots +a_{k}D_{k}$ with $a_{i}\in R$ and $a_{k}\not= 0.$ Then $$
[D_{i}, D]=[D_{i},(\sum _{i=1}^{k-1}a_{i}D_{i})+a_{k}D_{k}]=[D_{i}, (\sum _{i=1}^{k-1}a_{i}D_{i})]+[D_{i}, a_{k}D_{k}]=$$ $$=i_{1}+D_{i}(a_{k})D_{k}+a_{k}[D_{i}, D_{k}],   \ i=1, \ldots , k-1,$$
where $i_{1}=[D_{i}, \sum _{i=1}^{k-1}D_{i}]\in I$ by the choice of the  basis $\{ D_{1}, \ldots , D_{k}\} .$
Since $[D_{i}, D]\in I, i=1, \ldots , k-1$ we get  $D_{i}(a_{k})=0, \ i=1, \ldots , k-1.$  Further, we can assume without loss of generality  that $D_{k}\in J.$ As  $J/I$ is abelian  we see that $[D_{k}, a_{k}D_{k}]\in I$ and therefore $D_{k}(a_{k})=0.$ Then $a_{k}\in F$ and $\dim _{F}FJ/FI=1.$
In particular, if $\dim _{F}FL/FI=2,$ then $s(L/I)=2.$
If the quotient algebra $L/I$ is abelian, then it holds  $\dim _{F}FL/FI=1.$  Let $FL/FI$ be nonabelian. Then its derived subalgebra is abelian and therefore is one-dimensional over $F$ by Lemma \ref{solfewdim}. We may assume that $FD_{k}+FI/FI$ is the derived subalgebra of $FL/FI.$  If $\dim _{F}FL/FI>2,$ then there exists an element $i_{1}+aD_{k}\in FL\setminus (FD_{k}+FI)$ with $i_{1}\in RI$ and $a\in R$ such that $[D_{k}, i_{1}+a_{k}D_{k}]\in I.$ The latter  means that $D_{k}(a)=0$ and taking into account the equalities $D_{1}(a_{k})=0, i=1, \ldots , k-1$ obtained  analogously,  we see that $a_{k}\in F.$ But then  $i_{1}+a_{k}D_{k}\in FD_{k}+FI.$ The latter is impossible by the choice of this element and the obtained contradiction shows that $\dim _{F}FL/FI=2.$
\end{proof}

The next statement can be easily deduced from the Lie's Theorem for solvable Lie algebras over fields of characteristic zero and its modification over fields of positive characteristic (see, for example \cite{Hum}, Theorem 4.1 and Exercise 2 on p.20). We do not require the  ground field to be algebraically closed because one can always consider all the objects over the algebraic closure $\overline {\mathbb K}.$
\begin{lemma} \label{matrixalgebra}
Let $\mathbb K$ be a field, $V$ be a vector space of dimension $n$ over $\mathbb K$ and $L$ a solvable subalgebra of $gl_{n}(\mathbb K).$  If ${\rm char} \mathbb K=0$ or ${\rm char} \mathbb K>n,$ then $s(L)\leq n.$
\end{lemma}

\begin{theorem}\label{derlensol}
Let $L$ be a solvable subalgebra of rank $k$ over $R$ of the Lie algebra $W(A).$
If the ground field $\mathbb K$ is of  characteristic zero, then  the derived length s(L) of $L$ does not exceed $2k$. Moreover, if  $L$ is finite dimensional over its field   of constants,
then $s(L)\leq k+1.$
\end{theorem}

\begin{proof}
Since  $s(L)=s(FL)$ we  can assume  $L=FL$.
Let $J_1$ be an abelian ideal of $L$ of maximal rank over $R$, let  ${\rm rk}_R J_1=k_1$.
Take a basis $D_1, \dots , D_{k_1}$ of $J_1$ over $R$  and denote $I_1=RJ_1 \cap L$.
Then $I_1$ is also an ideal of $L$ by Lemma \ref{intersectionideal} and ${\rm rk}_{R}I_{1}=k_{1}.$
 Let $J_2/I_1$ be an abelian ideal of $L/I_1$ such that $J_{2}$ has maximal rank over $R.$ Denote
$I_2=RJ_2 \cap L.$
 Then $I_{2}$ is an ideal of $L$ of rank $k_{2}$ over $R.$ As above take a basis $D_{k_1 +1}, \dots , D_{k_2}$ of $J_2/I_1$. Continuing this consideration  we can construct the series of ideals:
\[0 \subset J_1 \subseteq  I_1  \subset \dots \subset J_s \subseteq  I_s=L,\]
with ${\rm rk}_R I_j={\rm rk}_R J_j =k_j$, $J_j / I_{j-1}$ is abelian, $I_j=RJ_j \cap L, j=1, \ldots , s.$
Simultaneously we obtain an R-basis $\lbrace D_1, \dots , D_{k_s} \rbrace$ of $L$ such that
$D_{k_{j-1} +1}, \dots , D_{k_j}$ is a basis of $J_j/I_{j-1}$, $j=1 \dots s$.

Let us prove the statement of Theorem  by induction on $s$. If $s=0$ then $L=\lbrace 0 \rbrace$ and the proof is completed.
Let $s \geq 1$. By the inductive assumption  $s(I_{s-1})\leq 2k_{s-1}.$
Let us show that the abelian ideal $J_s / I_{s-1}$ is of dimension $k_s-k_{s-1}$ over $F.$
Really,  for any element
$$D= c_1D_1 + \dots +c_{k_{s-1}} D_{k_{s-1}}+ c_{k_{s-1}+1} D_{k_{s-1}+1}+ \dots + c_{k_s}D_{k_s} \in J_s$$
we have $[D_j,D]\in I_{s-1}$, $j=1 \dots k_s$. One can write:
\[[D_j,D]=\sum_{i=k_{s-1}+1}^{k_s} (D_j(c_i)D_i +c_{i}[D_{j}, D_{i}])+ [D_j,\sum_{i=1}^{k_{s-1}}c_iD_i]. \]

Since $[D_{j}, D_{i}]\in I_{s-1}, i=k_{s-1}+1, \ldots , k_{s}$ and the second sum in the right side lies in $I_{s-1}$
we obtain that  $D_j(c_i)=0$, $j=1 \dots , k_s$, $i=k_{s-1} +1, \dots , k_s.$ Hence $c_i \in F$,
$i=k_{s-1} +1, \dots , k_s$
by definition of $F.$  Thus $\dim _{F}J_s/I_{s-1} =k_s -k_{s-1}$.

Note that we  have also proved that the centralizer of $J_s/I_{s-1}$ in the Lie algebra
$L/I_{s-1}$ coincides with $J_s/I_{s-1}$. Therefore
$L/J_s$ acts exactly  on the $F$-vector space $J_s/I_{s-1}$ of dimension $k_{s}-k_{s-1}$ over $F.$
 Since  $C_{L/J_{s}}(J_{s}/I_{s-1})=J_{s}/I_{s-1},$ then the solvable Lie algebra $L/J_{s}$ can be  be embedded  isomorphically into the general linear Lie algebra $gl_{k_{s}-k_{s-1}}(F).$  As solvable subalgebras of this Lie algebra have derived length $\leq k_{s}-k_{s-1}$ (by Lemma \ref{matrixalgebra}),  we see that $s(L/J_{s})\leq k_{s}-k_{s-1}.$  But then  $s(L)\leq 2k_{s-1}+ k_s-k_{s-1} \leq 2k_s=2k.$

If $L$ is finite dimensional over $F,$  then [L,L] is nilpotent of  derived length $\leq k$ by Corollary \ref{derlennil}. Therefore $s(L)\leq k+1$. This completes the proof of Theorem.
\end{proof}
\begin{remark}
The first part of Theorem \ref{derlensol} remains valid also in the case of positive characteristic of the ground field $\mathbb K$ provided that ${\rm char }\mathbb K>k$  (because its proof uses only Lemma \ref{matrixalgebra} with this restriction on the rank $k$).

\end{remark}
\begin{corollary}\label{maincor}
Let $\mathbb K$ be a field and $A$ be one of the following   algebras over  ${\mathbb K}:$
(1) $\mathbb K[x_{1}, \ldots , x_{n}]$ the polynomial algebra;
(2) $\mathbb K[[x_{1}, \ldots , x_{n}]]$ the algebra of formal power series;
(3) $\mathbb K(x_{1}, \ldots , x_{n})$ the field of rational functions;
(4) $\mathbb K((x_{1}, \ldots , x_{n}))$ the fraction field  of the algebra  $\mathbb K[[x_{1}, \ldots , x_{n}]].$
Let $\mathfrak{D}(A)$ be the Lie algebra of all $\mathbb K$-derivations $D$ of $A$ of the form $D=f_{1}\frac{\partial}{\partial x_{1}}+\cdots +f_{n}\frac{\partial}{\partial x_{n}}$ with $f_{i}\in A$ (in cases
(1) and (2) $\mathfrak{D}(A)$ obviously coincides with $\Der A$).
If $L$ is a nilpotent subalgebra of $\mathfrak{D}(A),$ then $L$ is finite dimensional over its field of constants and $s(L)\leq n.$ If $L$ is solvable and the ground field $\mathbb K$ is of characteristic zero, then $s(L)\leq 2n.$
\end{corollary}

Let $\mathbb K=\mathbb C,$  $A=\mathbb C[[x_{1}, \ldots , x_{n}]]$ , and  $\overline{W}_n(\mathbb{K})=\Der A$ be the Lie algebra of all vector fields with formal power series coefficients.
\begin{corollary}\label{Martelo}
Let $L$ be a nilpotent (solvable) subalgebra of $\overline{W}_n(\mathbb{K}).$
Then the derived length of $L$ does not exceed $n$ ($2n$ respectively).
\end{corollary}

 The last  statement  was proved recently  in \cite{Mart1}, where it was used to study groups of automorhisms of formal power series rings.
As the next example shows, the  bound in Theorem \ref{derlensol} cannot be improved (see also  \cite{Mart1}).

\begin{example}
Let $L=\lbrace \sum_{i=1}^n a_i \frac{\partial}{\partial x_i}  \in \overline{W}_n ({\mathbb{K}}) \  | \
a_j \in \mathbb{K}[[x_1, \dots x_{j-1}]]+x_j \mathbb{K}[[x_1, \dots x_{j-1}]]\rbrace$.
Then the    derived length  of $L$ equals  $2n$.
\end{example}

If $L$ is a solvable subalgebra of rank $2$ over $R$ of the Lie algebra $W(A),$  then $L$ is contained in a maximal (by inclusion) solvable subalgebra of rank $2$ over $R.$  Really, let $S_{2}$ be the set of all solvable subalgebras of rank $2$ over $R$  from $W(A).$ Using Theorem \ref{derlensol} one can easily show that the set $S_{2}$ is inductively ordered (by inclusion), so there exists by Zorn's Lemma at least one maximal element  of $S_{2}.$  The next statement shows the possible types for such maximal solvable subalgebras of rank $2$ over $R.$ Since any solvable subalgebra $L$ of rank $2$ over $R$ from  $W(A)$ is contained in a maximal subalgebra of the same type we get a rough characterization  of such Lie algebras.
\begin{proposition} \label{solv2}
Let $L$ be a  solvable subalgebra of $W(A)$  which is  maximal (by inclusion) among all solvable subalgebras of rank $2$ over $R$ from $W(A)$  and let $F=F(L)$ be its  field of constants.  If the ground field $\mathbb K$  is of characteristic zero, then $L$ is a Lie algebra over $F, $  the algebra $L$ contains elements $D_{1}, D_{2}$ with $[D_{2}, D_{1}]=aD_{1}$ for some $a\in F_{1}=\ker D_{1}$ and $L$ is one of the following algebras over the field $F:$

1. $ L=\langle D_{2}\rangle \rightthreetimes F_{1}D_{1}.$

2. $ L=\langle D_{2}\rangle \rightthreetimes (F_{1}D_{1}+bF_{1}D_{1}),$  where $b\in R, D_{1}(b)=1, $ $ D_{2}(b)=ab+a_{1}$ for some $a_{1}\in F_{1}.$

3. $L=(\langle D_{2}\rangle \rightthreetimes \langle cD_{1}+dD_{2}\rangle ) \rightthreetimes F_{1}D_{1},$ where $c\in R, d\in F_{1}$ such that $D_{1}(c)\in F_{1}, D_{2}(d)=1, $ $ D_{2}(c)=-ac+r$ for some $r\in F_{1}.$

4. $L=(\langle D_{2}\rangle \rightthreetimes \langle cD_{1}+dD_{2}\rangle ) \rightthreetimes (F_{1}D_{1}+F_{1}bD_{1}),$ where $D_{1}(b)=1,$ $D_{2}(d)=1, d\in F_{1}, D_{1}(c)\in F_{1},$ $D_{2}(c)=ac+r, $   $D_{2}(b)=ab+a_{1}$ for some $r, a_{1} \in F_{1}.$

5. L is isomorphic to a solvable subalgebra of the affine Lie algebra $ga_{2}(F)$ containing $F^{2},$ in particular $2\leq \dim _{F}L\leq 5.$
\end{proposition}
\begin{proof}
Let $L$ be a subalgebra of the Lie algebra $W(A)$ satisfying all the conditions of  this Proposition. Then $FL$ as a Lie algebra over the field $\mathbb K$ also satisfies these conditions and $L\subseteq FL.$  Therefore $FL=L$ because of maximality of  $L$  and $L$ is a Lie algebra over the field $F.$ We consider two cases dependent on  properties of  maximal abelian ideals of $L:$

{\underline{Case 1.}} Every   maximal abelian ideal of $L$ is of rank $1$ over $R.$ Take any two such ideals $I$ and $J$ of $L$ and let $D_{1}\in I, D_{2}\in J$ be nonzero elements. If $D_{1}$ and $D_{2}$ are linearly independent over $R,$ then
$I\cap J=0$ and $I+J$ is an abelian ideal of rank $2$ over $R$ from $L.$ But then $I+J$ is contained in a maximal abelian ideal of rank $2$ over $R$ from $L$ which contradicts to our assumption. Therefore $D_{1}$ and $D_{2}$ are linearly dependent over $R$ and $I+J$ is of rank $1$ over $R.$ Since $I+J$ is a nilpotent ideal of $L$ it follows from Proposition \ref{nilfewdim} that $I+J$ is abelian. But then $I=J$ and $I$ is the only maximal abelian ideal of  rank $1$ from $L.$ Denote $I_{1}=RI\cap L.$ The ideal $I_{1}$ has  rank $1$ over $R$  and $\dim _{F}L/I_{1}\leq 2$ by Lemma \ref{quotient}. Take any nonzero element $D_{1}$ from $I_{1}$ provided that $I_{1}$ is abelian, or from the abelian ideal $[I_{1}, I_{1}]$ in other case (recall that $I_{1}$ has derived length at most $2$). It can be easily shown  that $[D_{2}, D_{1}]=aD_{1}$ for some element $a\in F_{1}=\ker D_{1}$  and $F_{1}I_{1}$ is a subalgebra of $W(A)$ of rank $1$ over $R.$  It is easy to prove  that $[D_{2}, F_{1}I_{1}]\subseteq F_{1}I_{1}$ and therefore $L+F_{1}I_{1}$ is a solvable subalgebra of rank $2$ from $W(A).$
But then $L=L+F_{1}I_{1}$ because of maximality of $L$ and hence $F_{1}I_{1}\subseteq L.$ The latter means that $F_{1}I_{1}=I_{1}$ and $I_{1}$ is a Lie algebra over the field $F_{1}.$

{\underline{Subcase 1.}} The ideal $I_{1}$ is abelian.  If $\dim _{F}L/I_{1}=1,$ then choosing any element $D_{2}\in L\setminus I_{1}$ we see that  $L=\langle D_{2}\rangle \rightthreetimes F_{1}D_{1}$ is a Lie algebra of type 1. Let $\dim _{F}L/I_{1}=2.$  Then $L/I_{1}$ is nonabelian by Lemma \ref{quotient}. Take the one-dimensional ideal $\langle D_{2}+I_{1}/I_{1}$ from the quotient algebra $L/I_{1}.$ Take also any element $cD_{1}+dD_{2} \in L$ such that $[D_{2}, cD_{1}+dD_{2}]=D_{2}+rD_{1}$ for some element $rD_{1}\in I_{1}.$
This gives the equality $D_{2}(c)D_{1}+caD_{1}+D_{2}(d)D_{2}=D_{2}+rD_{1}$ which implies $D_{2}(d)=1,$ and $D_{2}(c)=-ac+r.$ Besides, from the inclusion
 $[D_{1}, cD_{1}+dD_{2}]\in I_{1}$ we get that $D_{1}(d)=0,$ i.e. $d\in F_{1}.$ The same relation also gives $D_{1}(c)\in F_{1}.$ We see that $L$ is a Lie algebra of type 3 of Proposition.

{\underline{Subcase 2.}} The ideal $I_{1}$ is nonabelian. Suppose  first that $\dim _{F}L/I_{1}=1$  and take any element $D_{2}\in L\setminus I_{1}.$ In view of Lemma \ref{solfewdim}, $I_{1}=F_{1}D_{1}+F_{1}bD_{1}$ for some $b\in R$ such that $D_{1}(b)=1.$ Since $[D_{2}, D_{1}]=aD_{1}$ for some $a\in F_{1},$ it holds $[D_{1}, D_{2}](b)=aD_{1}(b)=a.$  On the other hand  $(D_{1}D_{2}-D_{2}D_{1})(b)=D_{1}(D_{2}(b))=a.$ But then $D_{1}(ba-D_{2}(b))=a-a=0$ and hence $ba-D_{2}(b)\in F_{1}.$ Then $D_{2}(b)=ba+a_{1}$ for some element $a_{1}\in F_{1}$  and $L$ is a Lie algebra of type 2.
Let now $\dim _{F}L/I_{1}=2.$ The quotient algebra ${F}L/FI_{1}$ is nonabelian by Lemma \ref{quotient}. Take the one-dimensional ideal $\langle D_{2}+I_{1}\rangle$ from the quotient algebra $L/I_{1}$ (over $F$) and let $cD_{1}+dD_{2}$ be such an element that $[D_{2}, cD_{1}+dD_{2}]=D_{2}+rD_{1}$ for some element $rD_{1}\in I_{1}.$  It follows from this relation that  $D_{2}(d)=1$ and $D_{2}(d)=-ac+r$ for $r\in F_{1}.$ Further we have from the inclusion $[D_{1}, cD_{1}+dD_{2}]\in I_{1}$ that $D_{1}(d)=0.$ This means that $d\in F_{1}.$ Using the same  inclusion we get  $D_{1}(c)\in F_{1}.$ Further, as above one can show  that $D_{2}(b)=ab+a_{1}$ for some element $a_{1}\in F_{1}.$ and $L$ is a Lie algebra of type 4.

{\underline{Case 2.}} $L$ contains at least one  maximal abelian ideal  of rank $2$ over $R.$ Denote it by $J$ and
choose any two elements $D_{1}$ and $D_{2}$ from $J$ linearly independent over $R.$
If $D=u_{1}D_{1}+u_{2}D_{2}\in J,$ then from the equality $$0=[D_{i}, D]=[D_{i}, u_{1}D_{1}+u_{2}D_{2}]=D_{i}(u_{1})D_{1}+D_{i}(u_{2})D_{2}, i=1, 2$$
we obtain $D_{i}(u_{j})=0.$ The latter means that $u_{i}\in F$ i.e. $\dim _{F}J=2.$
Since $J$ is a maximal abelian ideal of $L$ it holds $C_{L}(J)=J.$ Therefore $\dim _{F}L/J\leq 3$ because of solvability of $L/J$ and equality $\dim J=2.$ Let us consider the case $\dim L/J=1$ and take any element $D_{3}\in L\setminus J.$ Then $D_{3}=u_{1}D_{1}+u_{2}D_{2}$ for some $u_1 , u_2 \in R.$  As
$$[D_i , D_3]=D_i(u_1)D_1+D_i(u_2)D_2\in J$$ we obtain $D_i(u_j)\in F, \  i, j=1,2.$ If the matrix
\begin{equation}\label{matrix}
\left(  \begin{array}{cc}  D_1(u_1)  & D_2(u_1) \\         D_1(u_2) & D_2(u_2) \\        \end{array}
   \right)
\end{equation}
   is nonsingular, then applying an appropriate linear transformation we can write
   $$u_1=\alpha _{11}v_1+\alpha _{12}v_2,   \   u_2= \alpha _{21}v_1+\alpha _{22}v_2  $$
   for some $\alpha _{ij} \in F$ and $D_{i}(v_{j})=\delta _{ij}, $ where $\delta _{ij}$ is the Kronecker symbol.
   It is obvious that $L_{1}=F\langle D_{1}, D_{2}, v_{i}D_{j} \ | \ i, j=1,2\rangle$ is a Lie algebra of dimension $6$ over $F$ isomorphic to the general affine Lie algebra  $ga_{2}(F).$  But then
   $$D_{3}=u_1D_1+u_2D_2=(\alpha_{11}v_1+\alpha _{12}v_2)D_1+(\alpha_{21}v_1+\alpha _{22}v_2)D_2$$ is an element of $L_{1}$ and $L$ is a subalgebra of $L_1.$

   Let now the matrix (\ref{matrix})  be degenerated. Since $D_{3}\in L\setminus J,$ at least one of the rows of the matrix (\ref{matrix}) is nonzero, let the first. Without loss of generality we can assume that $D_{1}(u_{1})=1, \ D_{2}(u_{1})=\gamma $ for some $\gamma \in F.$  The second row of the matrix (\ref{matrix}) is proportional to the first one and therefore $u_2=\alpha u_{1}+\beta $ for some $\alpha , \beta \in F.$ Then we have $D_{3}=u_{1}D_{1}+(\alpha u_{1}+\beta )D_{2}.$  Replacing the element $D_{3}$ by  the element $D_{3}-\beta D_2$ we can assume that $D_{3}=u_{1}D_{1}+\alpha u_{1}D_{2}.$ If $\gamma =0,$ then $D_{1}(u_{1})=1, D_{2}(u_{1})=0$ and $L$ is isomorphic to a subalgebra of $ga_{2}(F).$  In case $\gamma \not=0$ we choose the basis $D_{1}'=D_{1}, D_{2}'=D_{1}-\gamma ^{-1}D_{2}$ of the abelian ideal $J.$ Then we obtain $D_{1}'(u_{1})=1, D_{2}'(u_{1})=0$ and all is done. Analogously one can consider the cases $\dim L/J=2$ and $\dim L/J=3$ and show that $L$ is isomorphic to a subalgebra of $ga_{2}(F).$
\end{proof}

%%%%%%%%%%%%%%%%%%%%%%%%%%%%%%%%%%%%%%%%%%

%

\begin{thebibliography}{99}

\bibitem{Bavula}
V.V. Bavula, Lie algebras of unitriangular polynomial derivations and an isomorphism criterion for their Lie factor algebras. arXiv:math.RA:1204.4908.

 \bibitem{Draisma}
 J. Draisma, Transitive Lie algebras of vector fields: an overview.
Qual. Theory Dyn. Syst. 11 (2012), no. 1, 39-60.



\bibitem{Olver}
A. Gonz\'{a}lez-L\'{o}pez, N. Kamran and P.J. Olver, Lie algebras of
differential operators in two complex variavles, Amer. J. Math., 1992, v.114, 1163-1185.

\bibitem{Olver1}
A.  Gonz\'{a}lez-L\'{o}pez, N. Kamran and  P.J. Olver, Lie algebras of
vector fields in the real plane. Proc. London Math. Soc. (3) 64
(1992), no. 2, 339-368.

\bibitem{Hum}
J.E. Humphreys, Introduction to Lie Algebras and Representation
Theory, Springer Verlag, New York, 1972.
%
\bibitem{J1}
D.A.~Jordan, On the ideals of a Lie algebra of derivations.
J. London Math. Soc. (2) {\bf 33} (1986), no.~1, 33--39.

\bibitem{Lie} S. Lie,  Theorie der Transformationsgruppen, Vol. 3,
Leipzig,  1893.

\bibitem{MP12} Ievgen Makedonskyi and Anatoliy Petravchuk, On finite dimensional Lie algebras of planar vector fields with rational coefficients, submitted (arXiv:math.RA:1211.4165).

\bibitem{Mart1} M. Martelo and J. Ribon, Derived length of solvable groups of local diffeomorphisms, arXiv:math.DS:1108.5779.

\bibitem{PBNL}
R.O. Popovych, V.M. Boyko, M.O. Nesterenko  and M.W. Lutfullin,
Realizations of real low-dimensional Lie algebras, J. Phys. A 36
(2003), no. 26, 7337-7360.

\bibitem{Post}
G. Post, On the structure of graded transitive Lie algebras, J. Lie
Theory, 2002, V.12, N 1, 265--288.



\end{thebibliography}
\end{document}